\documentclass[a4paper,10pt]{amsart}
\usepackage[shortlabels]{enumitem}
\usepackage{amsthm,amsmath,amssymb,graphics,graphicx,hyperref,epstopdf,mathrsfs}
\usepackage{mathtools}
\usepackage{breqn}
\usepackage{tikz}
\usetikzlibrary{shapes.geometric}

\usetikzlibrary{calc,intersections,through}
  \usepackage{capt-of}

\title{Characterizing Steiner systems via Betti numbers of monomial ideals}
\author{Guillermo Alesandroni, Christopher Chin, Noah Ripke}
\address{}
\email{}

\newtheorem{theorem}{Theorem}[section]
\newtheorem{proposition}[theorem]{Proposition}
\newtheorem{corollary}[theorem]{Corollary}
\newtheorem{lemma}[theorem]{Lemma}

\newtheorem{conjecture}[theorem]{Conjecture}

\theoremstyle{definition}
\newtheorem{definition}[theorem]{Definition}
\newtheorem{remark}[theorem]{Remark}

\newtheorem{construction}[theorem]{Construction}

\DeclareMathOperator{\betti}{b}
\DeclareMathOperator{\pd}{pd}
\DeclareMathOperator{\mdeg}{mdeg}

\DeclareMathOperator{\lcm}{lcm}

\DeclareMathOperator{\hdeg}{hdeg}

\DeclareMathOperator{\rank}{rank}

\DeclarePairedDelimiter{\abs}{\lvert}{\rvert}

\begin{document}
\maketitle
\begin{abstract}
Given a point set $[n]$ and a collection $B$ of $k$-subsets of $[n]$, we create a monomial ideal whose Betti numbers completely determine whether the pair $([n],B)$ is a Steiner system $S(t,k,n)$. In particular, our construction characterizes Steiner triple systems and Steiner quadruple systems solely on the basis of Betti numbers of monomial ideals. Moreover, these monomial ideals allow us to reformulate the renowned prime power conjecture as follows:  If a monomial ideal $M$ of $S = K[x_1, \dots, x_{q^2+q+1}]$ generated by squarefree monomials of degree $q+1$ has Betti numbers $\betti_s(S/M) = {{q^2+q+1} \choose s} \text{ for all } s=0,\ldots,q$
then $q$ is a prime power.
\end{abstract}

\section{Preliminaries}
Throughout, the set of integers $\{1,\ldots,n\}$ will be denoted $[n]$. An element of $[n]$ will be called a \textbf{point} and a $k$-subset of $[n]$ will be called a \textbf{$k$-block}.

\begin{definition}
A \textbf{Steiner system} $S(t,k,n)$ is a pair $([n],B)$ where $B$ is a collection of $k$-blocks such that every $t$-block is contained in a unique $k$-block in $B$.
\end{definition}

Note that the number of $t$-blocks of $[n]$ is $n\choose t$, while the number of $t$-blocks of a $k$-block is $k\choose t$. In addition, if $([n],B)$ is an $S(t,k,n)$, then every $t$-block is contained in exactly one $k$-block in $B$ and thus, ${n\choose t}={k\choose t} \abs B$ or $\abs B={n\choose t}/{k\choose t}$.\\

$S(2,3,n)$ are called Steiner triple systems of order $n$, denoted $\mathrm{STS}(n)$. These are the most emblematic of all Steiner systems and have been studied in depth since 1850 when Thomas Kirkman proposed the now famous Kirkman's schoolgirl problem [Sp]. \newline \newline
Simple instances of Steiner triple systems are the projective plane of order 2 (the Fano plane) which is an $\mathrm{STS}(7)$, and the affine plane of order 3, which is an $\mathrm{STS}(9)$. Other systems of interest are the $S(2,q+1,q^2+q+1)$ because they are equivalent to projective planes of order $q$. In fact, the prime power conjecture can be formulated as follows: an $S(2,q+1,q^2+q+1)$ exists if and only if $q$ is a prime power. Finally, Steiner systems with parameters $t=3$ and $k=4$ are called Steiner quadruple systems, denoted $\mathrm{SQS}(n)$. These systems have also been extensively studied. For more information on Steiner systems, see [LR,VW]. \newline \newline
Now we introduce some standard definitions and constructions in commutative algebra. Let $S_n=K[x_1,\ldots,x_n]$ be a polynomial ring over an arbitrary field $K$ on $n$ variables, and denote by $M$ a squarefree monomial ideal in $S_n$. It follows from the Hilbert basis theorem that $M$ has a finite unique minimal generating set of monomials $m_1,\ldots,m_p$. Because of this fact, we will use the notation $M=(m_1,\ldots,m_p)$, where $m_1, \dots, m_p$ is the minimal generating set.

Now, we introduce the Taylor resolution [Ta], which will play a central role in the proofs of several results. For context and in-depth study, see [Me,Pe].

\begin{construction}
Let $M=(m_1,\ldots,m_p)$. For every subset $\{m_{i_1},\ldots,m_{i_s}\}$ of the minimal generating set, with $1\leq i_1<\ldots<i_s\leq p$, 
we create a formal symbol $[m_{i_1},\ldots,m_{i_s}]$, called a \textbf{Taylor symbol}. The Taylor symbol associated to $\emptyset$ is denoted by $[  ]$.
For each $s=0,\ldots,p$, set $F_s$ equal to the free $S$-module with basis $\{[m_{i_1},\ldots,m_{i_s}]:1\leq i_1<\ldots<i_s\leq p\}$ given by the 
${p\choose s}$ Taylor symbols corresponding to subsets of size $s$. That is, $F_s=\bigoplus\limits_{i_1<\ldots<i_s}S[m_{i_1},\ldots,m_{i_s}]$ 
(note that $F_0=S[ ]$). Define
\[f_0:F_0\rightarrow S/M\]
\[s[ ]\mapsto f_0(s[ ])=s.\]
For $s=1,\ldots,p$, let $f_s:F_s\rightarrow F_{s-1}$ be given by
\[f_s\left([m_{i_1},\ldots,m_{i_s}]\right)=
 \sum\limits_{j=1}^s\dfrac{(-1)^{j+1}\lcm(m_{i_1},\ldots,m_{i_s})}{\lcm(m_{i_1},\ldots,\widehat{m_{i_j}},\ldots,m_{i_s})}
 [m_{i_1},\ldots,\widehat{m_{i_j}},\ldots,m_{i_s}]\]
 and extended by linearity.
 The \textbf{Taylor resolution} $\mathbb{T}_M$ of $S/M$ is the exact sequence
 \[\mathbb{T}_M:0\rightarrow F_p\xrightarrow{f_p}F_{p-1}\rightarrow\cdots\rightarrow F_1\xrightarrow{f_1}F_0\xrightarrow{f_0} 
 S/M\rightarrow0.\]
 \end{construction}
 
 We define the \textbf{multidegree} of a Taylor symbol $[m_{i_1},\ldots,m_{i_s}]$, denoted $\mdeg[m_{i_1},\ldots,m_{i_s}]$, as 
  $\mdeg[m_{i_1},\ldots,m_{i_s}]=\lcm(m_{i_1},\ldots,m_{i_s})$. Similarly, we define the \textbf{degree} of $[m_{i_1},\ldots,m_{i_s}]$, denoted $\deg[m_{i_1},\ldots,m_{i_s}]$, as the degree of its multidegree.
 \\
 
\begin{definition}
Let
\[\mathbb{F}:\cdots\rightarrow F_i\xrightarrow{f_i}F_{i-1}\rightarrow\cdots\rightarrow F_1\xrightarrow{f_1}F_0\xrightarrow{f_0} S/M\rightarrow 0\]
be a free resolution of $S/M$. 
A basis element $[\sigma]$ of $\mathbb{F}$ has \textbf{homological degree} $i$, denoted $\hdeg[\sigma]=i$, if 
$[\sigma] \in F_i$. If all the differential matrices $\left(f_i\right)$ of $\mathbb{F}$ contain only noninvertible entries, then $\mathbb{F}$ is called the \textbf{minimal resolution} of $S/M$.
\end{definition} 
The minimal resolution of $S/M$ is unique up to isomorphism and it can be obtained from $\mathbb{T}_M$ by means of consecutive cancellations. The process of consecutive cancellations is technical and it will not be described in this article. However, the interested reader can find a rigorous treatment in [Al,Pe].

\begin{definition}
Let
 \[\mathbb{F}:\cdots\rightarrow F_i\xrightarrow{f_i}F_{i-1}\rightarrow\cdots\rightarrow F_1\xrightarrow{f_1}F_0\xrightarrow{f_0} S/M\rightarrow 0\]
be the minimal free resolution of $S/M$.
\begin{itemize}
 \item For every $i\geq 0$, the $i^{th}$ \textbf{Betti number} $\betti_i\left(S/M\right)$ of $S/M$ is $\betti_i\left(S/M\right)=\rank(F_i)$.
 \item For every $i,j\geq 0$, the \textbf{graded Betti number} $\betti_{i,j}\left(S/M\right)$ of $S/M$, in homological degree $i$ and internal degree $j$,
is \[\betti_{i,j}\left(S/M\right)=\#\{\text{basis elements }[\sigma]\text{ of }F_i:\deg[\sigma]=j\}.\]
\item For every $i\geq 0$, and every monomial $m$, the \textbf{multigraded Betti number} $\betti_{i,m}\left(S/M\right)$ of $S/M$, in homological degree $i$ and multidegree $m$,
is \[\betti_{i,m}\left(S/M\right)=\#\{\text{basis elements }[\sigma]\text{ of }F_i:\mdeg[\sigma]=m\}.\]
\item The \textbf{projective dimension} $\pd\left(S/M\right)$ of $S/M$ is \[\pd\left(S/M\right)=\max\{i:\betti_i\left(S/M\right)\neq 0\}.\]
\end{itemize}
\end{definition}

\begin{definition}\label{Def 2}
Let $M=(m_1,\ldots,m_p)$ be a squarefree monomial ideal of $S_n$. We will say that $M$ is \textbf{$r$-coverfree} if no minimal generator $m_i$ divides the $\lcm$ of at most $r$ others. (Another common definition of this concept uses ``exactly $r$ others" instead. These definitions are equivalent as long as $p \geq r+1$.)
\end{definition}

For example, $M=(x_1x_2x_5,x_1x_3x_6,x_2x_3x_4,x_4x_5x_6)$ is 2-coverfree but not 3-coverfree, as $x_1x_2x_5\mid \lcm(x_1x_3x_6,x_2x_3x_4,x_4x_5x_6)$. In the case when $M=(m_1,\ldots,m_p)$ is $(p-1)$-coverfree, the definition coincides with the concept of a dominant ideal. It is proven in [Al] that the Taylor resolution of a squarefree monomial ideal is minimal if and only if $M$ is dominant. The above definition is closely related to the next one.

\begin{definition}\label{Def 3}
Let $G=\{m_1,\ldots,m_p\}$ be a collection of sets. We will say that $G$ is \textbf{$r$-coverfree} if no $m_i$ is contained in the union of at most $r$ others.
\end{definition}

For example, $G=\left\{\{x_1,x_2,x_5\},\{x_1,x_3,x_6\},\{x_2,x_3,x_4\},\{x_4,x_5,x_6\}\right\}$ is 2-coverfree but not 3-coverfree, as $\{x_1,x_2,x_5\}\subseteq \{x_1,x_3,x_6\}\cup\{x_2,x_3,x_4\}\cup\{x_4,x_5,x_6\}$.
\newline \newline
Note that if $G$ is $r$-coverfree, it is certainly $s$-coverfree for any $s \leq r$.
\newline \newline
Coverfree collections of sets were first studied equivalently as binary disjunct matrices before being studied as classical objects in combinatorics. For a survey of definitions, constructions, and results related to coverfree collections of sets, we refer to [We].
\newline \newline
The symbols $m_i$ in Definitions \ref{Def 2} and \ref{Def 3} represent monomials and sets, respectively. This choice is intentional because if $M=(m_1,\ldots,m_p)$ and if we replace $M$ with $G$, $m_i$ with the set of variables that appear in its factorization (also denoted $m_i$), and $\lcm$ with union, then we see that $M$ is $r$-coverfree if and only if $G$ is $r$-coverfree. There is another reason for viewing the symbols $m_i$ as monomials as well as sets: in the next definition, we will identify $k$-blocks $\{i_1,\ldots,i_k\}$ of $[n]$ with monomials $x_{i_1}\ldots x_{i_k}$ of $S_n$. To emphasize the analogy between these objects, both will be denoted with the same symbol. 

\begin{definition}\label{Def 4}
Given the pair $([n],B)$, where $B=\{m_1,\ldots,m_p\}$ is an arbitrary collection of $k$-blocks, we define a monomial ideal $M_B$ of $S_n$ as follows: $x_{i_1}\cdots x_{i_k}$ is a minimal generator of $M_B$ if and only if $\{i_1,\ldots,i_k\}$ is a $k$-block in $B$. The ideal $M_B$ will be called the \textbf{block ideal} of $B$. In particular, if $([n],B)$ is a Steiner system, the block ideal $M_B$ will be called the \textbf{Steiner ideal} of $B$.
\end{definition}
We note that the above definition has been described in [EGE]. Similarly, a definition based on hypergraphs appears in [HT]. But our work with Steiner ideals is distinct from both papers.
\section{Characterization of $\mathrm{STS}(n)$}

In this section we will show that $\betti_1(S_n/M_B)$ and $\betti_2(S_n/M_B)$ determine whether $([n],B)$ is an $\mathrm{STS}(n)$.

\begin{lemma}\label{Lemma 5}
Let $M=(m_1,\ldots,m_p)\subset S_n$. Then $M$ is $r$-coverfree if and only if
\[
\betti_s(S_n/M)=\binom{p}{s} \text{ for } s = 0, 1, \dots, r.
\]
Additionally, if we also have $p \geq r+1$, then in fact $M$ is $r$-coverfree if and only if 
\[
\betti_r(S_n/M)=\binom{p}{r}.
\]
\end{lemma}
\begin{proof}
First suppose $M$ is $r$-coverfree. Let $[\sigma]=[m_{i_1},\ldots,m_{i_{s+1}}]$ and $[\tau]=$\\$[m_{i_1},\ldots,\widehat{m_{i_j}},\ldots,m_{i_{s+1}}]$ be basis elements of $\mathbb{T}_M$ in homological degrees $s+1$ and $s$, respectively, for $0 \leq s \leq r$. Then $m_{i_j}\nmid \mdeg [\tau]$ and hence $\mdeg [\sigma] \neq \mdeg [\tau]$. Now, the differential map $f_{s+1}$ of $\mathbb{T}_M$ acts on the basis element $[\sigma]$ as follows: $f_{s+1}([\sigma]) = \cdots + \frac{(-1)^{j+1}\lcm(m_{i_1},\ldots,m_{i_{s+1}})}{\lcm(m_{i_1},\ldots,\widehat{m_{i_j}},\ldots,m_{i_{s+1}})}[\tau] + \cdots$. 
\newline \newline
This means that the coefficient of $[\tau]$ in the representation of $f_{s+1}([\sigma])$ as a linear combination of the basis elements of $F_s$ is noninvertible. We conclude that the differential matrix $(f_{s+1})$ has no invertible entries. Therefore, the minimal resolution of $S_n/M$, obtained from $\mathbb{T}_M$ by means of consecutive cancellations has the same free $S$-module as $\mathbb{T}_M$ in homological degree $s$. In particular, $\betti_s(S_n/M) = \rank(F_s) = \binom{p}{s}$.
\newline \newline
On the other hand, if $M$ is not $r$-coverfree, then in fact there is some minimal generator which divides the $\lcm$ of a collection of $s \leq r$ others, so in fact there are some basis elements $[\sigma]=[m_{i_1},\ldots,m_{i_{s+1}}]$ and $[\tau]=[m_{i_1},\ldots,\widehat{m_{i_j}},\ldots,m_{i_{s+1}}]$ where $m_{i_j}\mid \mdeg[\tau]$ and hence $\mdeg[\sigma] = \mdeg [\tau]$. This means that the coefficient of $[\tau]$ in the representation of $f_{s+1}([\sigma])$ as a linear combination of the basis elements of $F_s$ is $\pm 1$, invertible. Thus at least one consecutive cancellation can be performed in $\mathbb T_M$ in homological degree $s$. In particular, passing to the minimal resolution, $\betti_s(S_n/M) < \rank(F_s)= \binom{p}{s}$.
\newline \newline
If $p \geq r+1$, instead add arbitrary minimal generators to this collection if necessary to get a collection of exactly $r$ others. Repeat the above argument with the basis elements associated to this collection, say, $[\sigma']=[m_{i_1},\ldots,m_{i_{r+1}}]$ and $[\tau']=[m_{i_1},\ldots,\widehat{m_{i_j}},\ldots,m_{i_{r+1}}]$ to get that $\betti_r(S_n/M) < \rank(F_r)= \binom{p}{r}$.
\end{proof}

\begin{lemma}\label{Lemma 6}
If $([n],B)$ is an $S(2,k,n)$, then $B$ is $(k-1)$-coverfree.
\end{lemma}
\begin{proof}
Let $m,m_1,\ldots,m_{k-1}$ be $k$ distinct $k$-blocks in $B$. Since any pair of $k$-blocks can share at most one point, we have that
\[\abs*{m\cap \bigcup\limits_{i=1}^{k-1}m_i} = \abs*{\bigcup\limits_{i=1}^{k-1}(m\cap m_i)}\leq \sum\limits_{i=1}^{k-1} \abs{m\cap m_i} \leq k-1.\]
Since $\abs m = k$, it follows that $m\nsubseteq \bigcup\limits_{i=1}^{k-1} m_i$. Thus, $B$ is $(k-1)$-coverfree.
\end{proof}

\begin{theorem}\label{Th 7}
Let $([n],B)$ be an $S(2,k,n)$ and let $M_B$ be the Steiner ideal of $B$. Then for all $s=0,\ldots,k-1$, $\betti_s(S_n/M_B) = \dbinom{\frac{n(n-1)}{k(k-1)}} {s}$.
\end{theorem}

\begin{proof}
Note that $\abs B = \frac{n(n-1)}{k(k-1)}$ and by Lemma \ref{Lemma 6}, $B$ is $(k-1)$-coverfree.\\
Equivalently, $M_B$ is minimally generated by $\frac{n(n-1)}{k(k-1)}$ monomials and it is $(k-1)$-coverfree. Now the result follows from Lemma \ref{Lemma 5}. 
\end{proof}

Later, we will prove that the converse of Theorem \ref{Th 7} holds. That is, if the Betti numbers of the block ideal of $B$ have the right binomial coefficients, then $([n],B)$ is an $S(2,k,n)$. 

\begin{definition}\label{Def ???}
Given a collection $B$ of $k$-blocks of $[n]$, we will say that a $t$-block of $[n]$ is a \textbf{private set} if it is contained in exactly one $k$-block in $B$. In addition, we will say that a $t$-block is a \textbf{public set} if it is contained in more than one $k$-block in $B$. Similarly, a point $i$ will be called a \textbf{private point} or a \textbf{public point} if $\{i\}$ is private or public, respectively. Finally, a $t$-block will be called \textbf{uncovered} if it is not contained in any of the $k$-blocks in $B$. 
\end{definition}

\begin{lemma}\label{Lemma 8}
Suppose that $B$ is a $(k-1)$-coverfree set of $k$-blocks of $[n]$, and let $m\in B$. If some $2$-block of $m$ is public, then $m$ has a private point.
\end{lemma}

\begin{proof}
Suppose $m$ and $m'$ share the same $2$-block $\{a,b\}$. Then $m=\{a,b,i_1,\ldots,i_{k-2}\}$ and  $m'=\{a,b,j_1,\ldots,j_{k-2}\}$. By means of contradiction, assume that for all $s=1,\ldots,k-2$ there is a $k$-block $m_s\in B\setminus\{m\}$ such that $i_s\in m_s$. Then $m\subseteq \bigg (\bigcup\limits_{s=1}^{k-2} m_s \bigg) \cup m'$, which contradicts the fact that $B$ is $(k-1)$-coverfree. Thus, some point $i_s$ in $m$ must be private.
\end{proof}

\begin{lemma}\label{Lemma 9}
Suppose that $B$ is a $2$-coverfree set of $3$-blocks of $[n]$. Then, $\abs B \leq n(n-1)/6$. Furthermore, if $n>4$, equality holds only when no pair of $3$-blocks in $B$ share a $2$-block.
\end{lemma}

\begin{proof}
If $n\leq 3$, the result holds trivially. We will prove the case $n\geq 4$ by induction. 
\newline \newline
Suppose that the lemma holds for $n-1$, $n \geq 4$, and let $B$ be a 2-coverfree collection of 3-blocks of $[n]$. On the one hand, if no pair of 3-blocks in $B$ share a 2-block, then each 3-block in $B$ contains 3 unique 2-blocks. There are $n\choose 2$ distinct $2$-blocks; so we must have that $3 \abs B \leq {n\choose 2}$. That is, $\abs B \leq n(n-1)/6$. \newline \newline
On the other hand, if two 3-blocks of $B$ share a 2-block, it follows from Lemma \ref{Lemma 8} that some 3-block in $B$ contains a private point. By removing that point from $[n]$ and that 3-block from $B$, we obtain a collection $B'$ of 3-blocks on the remaining $n-1$ points. By our inductive hypothesis,
\[\abs B = \abs{B'}+1 \leq \frac{(n-1)(n-2)}{6}+1 = \frac{n(n-1)}{6} - \frac{2(n-4)}{6} \leq \frac{n(n-1)}{6}\]
for $n \geq 4$. In fact this inequality becomes strict for $n > 4$. 

Therefore, if $\abs B = n(n-1)/6$ and $n>4$, we must have that no pair of 3-blocks in $B$ share a 2-block.
\end{proof}
\begin{corollary}\label{Cor 10}
Suppose that $B$ is a collection of $3$-blocks of $[n]$, with $\abs B =n(n-1)/6$, for $n > 4$. If $B$ is $2$-coverfree, then $([n],B)$ is an $\mathrm{STS}(n)$.
\end{corollary}

\begin{proof}
By Lemma \ref{Lemma 9}, no pair of 3-blocks in $B$ share a 2-block. Since $\abs B = n(n-1)/6$, it follows that $3 \abs B = \binom{n}{2}$ and hence, each 2-block must be in a unique 3-block. In other words, $([n],B)$ is an $\mathrm{STS}(n)$.
\end{proof}
\vspace{\baselineskip}
\begin{theorem}\label{Th 11}
Consider the pair $([n],B)$ where $B$ is a collection of 3-blocks of $[n]$, for $n>4$, and let $M_B$ be the block ideal of $B$. Then $([n],B)$ is an $\mathrm{STS}(n)$ if and only if $\betti_1(S_n/M_B) = n(n-1)/6$ and $\betti_2(S_n/M_B) = \dbinom{n(n-1)/6} {2}$.
\end{theorem}

\begin{proof}
The forward implication follows immediately from Theorem \ref{Th 7}. Now, we will prove that if $([n],B)$ is not an $\mathrm{STS}(n)$, then $\betti_1(S_n/M_B) \neq \frac{n(n-1)}{6}$ or $\betti_2(S_n/M_B) \neq \binom{n(n-1)/6}{2}$. \newline \newline
If $\abs B \neq n(n-1)/6$, then $\betti_1(S_n/M_B) \neq n(n-1)/6$ and we are done. If $\abs B = n(n-1)/6$ but $([n],B)$ is not an $\mathrm{STS}(n)$, then $B$ is not 2-coverfree, by Corollary \ref{Cor 10}, and note that $\abs B \geq 3$. It follows that $M_B$ is not 2-coverfree, and by Lemma \ref{Lemma 5}, we have that $\betti_2(S_n/M_B) < \rank(F_2) = {{n(n-1)/6}\choose 2}$.
\end{proof}

\section{Characterization of $S(2,k,n)$}
We extend our methods to general $S(2,k,n)$ systems. First, we prove several lemmas.
\begin{lemma}\label{Lemma A}
If $k \geq 4$ and $k+2\leq r\leq (k-1)^2$, then $(r-k+2){k\choose 2}>{r\choose 2}$.
\end{lemma}

\begin{proof}
Let $f:[k+2,(k-1)^2]\rightarrow \mathbb{R}$ be defined by $f(x)=(x-k+2)k(k-1)-x(x-1)$. \newline
Equivalently, $f(x)=-x^2+(1+k(k-1))x+(2-k)k(k-1)$; in particular, $f$ is quadratic and concave down. Thus, it achieves its absolute minimum value at one of the endpoints of $[k+2,(k-1)^2]$. Now, $f(k+2)=3k^2-7k-2>0$, and $f\left((k-1)^2\right)=k^2-k>0$ for any $k \geq 4$. This means that $f(x)>0$ for all $x\in [k+2,(k-1)^2]$. In particular, when $r\in \mathbb{N}$ and $r\in [k+2,(k-1)^2]$, we have that:
\[
0<f(r)=(r-k+2)k(k-1)-r(r-1)=2\big[(r-k+2)\tbinom{k}{2}-\tbinom{r}{2}\big]
\]
which, in turn, implies that $(r-k+2){k\choose 2}>{r\choose 2}$.
\end{proof} 

\begin{lemma}\label{Lemma B}
Suppose that $B$ is a $(k-1)$-coverfree set of $k$-blocks of $[n]$, with $k\geq 3$ and $n=(k-1)^2$. Then, $\abs B \leq \frac{n(n-1)}{k(k-1)}$.
\end{lemma}

\begin{proof}
The $k=3$ case follows from Lemma \ref{Lemma 9}. Now suppose that $k \geq 4$. Since
\[
\frac{n(n-1)}{k(k-1)} = \dfrac{(k-1)^2[(k-1)^2-1]}{k(k-1)}=\frac{(k-1)(k-2)k}{k} = (k-1)(k-2)=n-k+1,
\]
we need to show that $\abs B \leq n-k+1$. Assume, by way of contradiction, that $\abs B \geq n-k+2$.
\newline \newline
For $k+2 \leq j \leq n$, let $B_j \subseteq B$ be a collection of at least $j-k+2$ different $k$-blocks in $B$, where all the blocks in $B_j$ live on at most $j$ of the points in $[n]$. (Such a collection exists for $j=n$; $B$ itself is one such example of a $B_n$.)
\newline \newline
The $k$-blocks of $B_j$ capture $(j-k+2){k\choose 2}$ 2-blocks, when there are only $\binom{j}{2}$ distinct 2-blocks on $j$ points. By Lemma \ref{Lemma A}, $(j-k+2){k\choose 2}>{j\choose 2}$ and thus, there must be two $k$-blocks in $B_j$ which contain the same 2-block. By Lemma \ref{Lemma 8}, there is a $k$-block in $B_j$ which has a private point. 
\newline \newline
Now, we will derive a contradiction as follows. Given that some $B_j$ exists, where $k+2\leq j\leq n$ and $B_j$ contains at least $j-k+2$ blocks on at most $j$ points, apply the previous reasoning to conclude that there is a $k$-block in $B_j$ which contains a private point. Remove that $k$-block and that point to obtain a collection $B_{j-1}$, which consists of at least $(j-1)-k+2$ different $k$-blocks on at most $j-1$ points.
\newline \newline
Thus, starting with $B_n := B$, we construct a sequence of collections $B_n, B_{n-1}, \dots, B_{k+1}$. The last collection, $B_{k+1}$, consists of at least $(k+1)-k+2=3$ different $k$-blocks on $k+1$ points. Since any two $k$-blocks $m_1,m_2\in B_{k+1}$ cover the $k+1$ points, a third $k$-block $m_3$ must be contained in the union of the other two. Returning to $B$, we have that the collection $\{m_1, m_2\} \subseteq B \setminus \{m_3\}$ covers $m_3$, contradicting $(k-1)$-coverfreeness, so we are done.
\end{proof}

\begin{theorem} \label{Theorem C}
Suppose that $B$ is a $(k-1)$-coverfree set of $k$-blocks of $[n]$, with $k\geq 3$ and $n\geq (k-1)^2$. Then $\abs B \leq \frac{n(n-1)}{k(k-1)}$. Furthermore, if $n>(k-1)^2$ and $\abs B =\frac{n(n-1)}{k(k-1)}$, then $([n],B)$ is an $S(2,k,n)$.
\end{theorem}

\begin{proof}
By induction on $n$. If $n=(k-1)^2$, then the statement holds by Lemma \ref{Lemma B}. Now, suppose that $n>(k-1)^2$ and the result holds for $n-1$. \newline \newline
On one hand, if no two blocks in $B$ share a 2-block, then there are no repetitions and the number of 2-blocks captured by all the $k$-blocks in $B$ is at most ${n\choose 2}$; that is, $\abs B {k\choose 2}\leq {n\choose 2}$. Therefore, $\abs B \leq \frac{n(n-1)}{k(k-1)}$. 
\newline \newline
On the other hand, if some $2$-block is public, then by Lemma \ref{Lemma 8} there is a $k$-block in $B$ which contains a private point. Remove that $k$-block and that point to obtain a collection $B'\subseteq B$ on at most $n-1$ points. By our inductive hypothesis,
\[
\abs B = \abs{B'}+1 \leq \frac{(n-1)(n-2)}{k(k-1)}+1 = \frac{n(n-1)}{k(k-1)} - \frac{2(n-1)-k(k-1)}{k(k-1)},
\]
and for $n > (k-1)^2$,
\[
\abs B < \frac{n(n-1)}{k(k-1)} - \frac{2((k-1)^2-1)-k(k-1)}{k(k-1)} = \frac{n(n-1)}{k(k-1)} - \frac{k-3}{k-1}.
\]
In particular, we have $\abs B < \frac{n(n-1)}{k(k-1)}$, as desired. Note that the inequality is strict in this case. 
\newline \newline
Now for $n > (k-1)^2$, equality only holds in the case where no two $k$-blocks in $B$ share a 2-block. But since $\abs B = \frac{n(n-1)}{k(k-1)}$, each 2-block of $[n]$ must be contained in a unique $k$-block in $B$. We conclude in this case that $([n],B)$ is an $S(2,k,n)$.
\end{proof}

\begin{remark}
Theorem \ref{Theorem C} may be regarded as a particular case of the excellent article [EFF], by Erdös et al. However, Theorem \ref{Theorem C} holds true for all $n>(k-1)^2$, while the result in [EFF] was proved asymptotically, without an explicit lower bound on $n$. This seemingly minor difference is in reality critical. For example, to reformulate the prime power conjecture we will use the fact that $n>(k-1)^2$ (and we would not be able to restate that conjecture if we replaced the constraint $n>(k-1)^2$ with, say, $n>k^2$ in Theorem \ref{Theorem C}).
\end{remark}
Therefore, although we acknowledge the work in [EFF], the upcoming results rely on our own Theorem \ref{Theorem C}.

\begin{theorem}\label{Th D}
Suppose that $k\geq 3$ and $n>(k-1)^2$. Consider a collection $B$ of $k$-blocks of $[n]$, and let $M_B$ be the block ideal of $B$. Then $([n],B)$ is an $S(2,k,n)$ if and only if $\betti_1(S_n/M_B) = \dfrac{n(n-1)}{k(k-1)}$ and $\betti_{k-1}(S_n/M_B) = \dbinom{\frac{n(n-1)}{k(k-1)}} {k-1}$.
\end{theorem} 

\begin{proof}
The forward implication follows from Theorem \ref{Th 7}. Now, we will show that if $([n],B)$ is not an $S(2,k,n)$, then $\betti_1(S_n/M_B) \neq \frac{n(n-1)}{k(k-1)}$ or $\betti_{k-1}(S_n/M_B) \neq \dbinom{\frac{n(n-1)}{k(k-1)}} {k-1}$. 
\newline \newline
If $\abs B \neq \frac{n(n-1)}{k(k-1)}$, then $\betti_1(S_n/M_B) \neq \frac{n(n-1)}{k(k-1)}$. Now assume $\abs B = \frac{n(n-1)}{k(k-1)}$ but $([n],B)$ still fails to be an $S(2,k,n)$; it follows from Theorem \ref{Theorem C} that $B$ is not $(k-1)$-coverfree and thus $M_B$ is not $(k-1)$-coverfree. Also note that 
\[
\abs B = p = \frac{n(n-1)}{k(k-1)}>(k-2)(k-1) \geq k-1.
\]
In particular $p \geq k$, so by Lemma \ref{Lemma 5}, we have $\betti_{k-1}(S_n/M_B) < \binom{\abs B}{k-1}$, and we are done.
\end{proof}

Since a Steiner system $S(2,q+1,q^2+q+1)$ is equivalent to a projective plane of order $q$, the prime power conjecture can be formulated purely in terms of Betti numbers of monomial ideals as follows:
\begin{conjecture}[Algebraic Prime Power Conjecture]
If a monomial ideal $M$ of $S = K[x_1, \dots, x_{q^2+q+1}]$ generated by squarefree monomials of degree $q+1$ has Betti numbers
\[
\betti_1(S/M) = q^2+q+1, \qquad \betti_q(S/M) = \binom{q^2+q+1}{q}
\]
then $q$ is a prime power.
\end{conjecture}

(Note that in this case $k=q+1$ and $n=q^2+q+1$. Therefore, the condition $n>(k-1)^2$ stated in Theorem \ref{Th D} is satisfied. Also, clearly $q^2+q+1 \geq q+1$, so we only need to check $\betti_1(S/M)$ and $\betti_q(S/M)$.)

\section{Characterization of $\mathrm{SQS}(n)$}

The notation that we introduce below is more general than what this section requires, but it is appropriate for the next section.
\begin{definition}
Let $m$ be a set, and let $0 \leq s\leq \abs m$. Then $\binom{m}{s}$ is the collection of all $s$-element subsets of $m$.
\end{definition}
\begin{definition}
Let $B$ be a collection of $k$-blocks of $[n]$. We define the \textbf{$s$-shadow design} of $B$, denoted $B^{(s)}$, and the \textbf{$s$-shadow ideal} of $B$, denoted $M_B^{(s)}$, as 
\[B^{(s)} = \left\{{m\choose s} : m\in B\right\}\quad \text{and} \quad M_B^{(s)} = \left(\prod_{S \in \binom{m}{s}} x_S : m \in B\right).\]
\end{definition}

(Note that $M_B^{(s)}$ is generated by monomials of degree $k\choose s$, and it agrees with the block ideal $M_B$ when $s=1$.)

For example, if $n=6$, $B=\Bigl\{\{1,2,3\},\{2,3,6\}\Bigr\}$, and $s=2$, then \\
$B^{(s)} = \biggl\{\Bigl\{\{1,2\},\{2,3\},\{1,3\}\Bigr\},\Bigl\{\{2,3\},\{3,6\},\{2,6\}\Bigr\}\bigg\}$ and \\
$M_B^{(s)} =
 \Bigl(x_{\{1,2\}}x_{\{2,3\}}x_{\{1,3\}}, x_{\{2,3\}}x_{\{3,6\}}x_{\{2,6\}} \Bigr)$
 
 \textit{Note}: When we write the quotient $S/M_B^{(s)}$, it will be assumed that $S$ is the polynomial ring $S=K\left[x_S : S\in{{[n]} \choose s}\right]$ so that the quotient makes sense.
 
 \begin{proposition}\label{Prop 2'}
 Suppose that $([n],B)$ is an $S(t,k,n)$, $2 \leq t<k<n$, and let $p={n\choose t}/{k\choose t}$. Then $\betti_i\left( S/M_B^{(t-1)}\right) = {p\choose i}$, for all $i=0,\ldots,{k\choose{t-1}} - 1$.
 \end{proposition}

\begin{proof}
Suppose that $m_1,m_2$ are distinct generators of $M_B^{(t-1)}$ which share 2 variables (indexed by $(t-1)$-blocks). Then, when viewed as $k$-blocks of $B$, $m_1$ and $m_2$ have 2 $(t-1)$-blocks in common. This means that $\abs{m_1\cap m_2} \geq t$, which contradicts the fact that $([n],B)$ is an $S(t,k,n)$. Therefore, two generators of $M_B^{(t-1)}$ can share at most one variable. Since each generator is of degree ${k\choose {t-1}}$, it follows that $M_B^{(t-1)}$ is $r$-coverfree for $r={k\choose {t-1}} -1$. Also, since $\abs B = {n\choose t}/{k\choose t}$, we have that $M_B^{(t-1)}$ is minimally generated by ${n\choose t}/{k\choose t}$ monomials. Now, the result follows from Lemma \ref{Lemma 5}.
\end{proof}

\begin{lemma}\label{Lemma 3'}
Let $B\subseteq {{[n]}\choose 4}$ and suppose that $B^{(2)}$ is 4-coverfree. If $m,m'\in B$ and $\abs {m \cap m'} = 3$, then $m$ has a private $2$-block.
\end{lemma}
\begin{proof}
Let $a,b,c\in [n]$ be such that $m\cap m' = \{a,b,c\}$, and denote $m=\{a,b,c,d\}$. Thus, a private 2-block must be in $\Bigl\{ \{a,d\},\{b,d\},\{c,d\}\Bigr\}$. Suppose by means of contradiction that each of $\{a,d\},\{b,d\},\{c,d\}$ is public in $B^{(2)}$. Then, there are (not necessarily distinct) $4$-blocks $l, l', l'' \in B\setminus\{m\}$ containing $\{a,d\},\{b,d\},\{c,d\}$, respectively. This means that ${m\choose 2} \subseteq {m'\choose 2} \cup {l\choose 2} \cup {l'\choose 2} \cup {l''\choose 2}$, contradicting the fact that $B^{(2)}$ is 4-coverfree. Thus, $m$ must have some private $2$-block.
\end{proof}

\begin{theorem}\label{Th 4'}
Let $B\subseteq {{[n]}\choose 4}$ and suppose that $B^{(2)}$ is 4-coverfree. Then for $n> 10$, we have $\abs B \leq {n\choose 3}/4$, with equality holding if and only if $([n],B)$ is an $S(3,4,n)$.
\end{theorem}

\begin{proof}
The backward direction is trivial; if $([n], B)$ is an $S(3,4,n)$, two blocks in $B$ may intersect in at most one 2-block; each block has six 2-blocks, so $B^{(2)}$ is in fact 5-coverfree. By counting how many 3-blocks each 4-block in $B$ captures, we also have $4\abs B = \binom{n}{3} \implies \abs B = \binom{n}{3}/4$.

Now we will prove the inequality. For each $l\subseteq [n]$, define $\lambda(l) = \abs{\{m\in B: l\subseteq m\}}$. Also define 
\[U = \abs*{\left\{ l\in \binom{[n]}{3}:\lambda(l) = 0 \right\}}, \qquad E = \sum\limits_{l \in \binom{[n]}{3}} \max(\lambda(l)-1,0).\]
That is, $U$ is the number of 3-blocks which are uncovered, and $E$ is the number of times that 3-blocks are excessively covered. Note that $\sum\limits_{l\in \binom{[n]}{3}} \lambda(l) - E = {n\choose 3} - U$. Also notice that a 4-block in $B$ covers four different 3-blocks, which implies that $4 \abs B = \sum\limits_{l\in \binom{[n]}{3}}\lambda(l) = {n\choose 3} - U+E$.

Now, define $P = \abs*{\left\{ e\in \binom{[n]}{2}:\lambda(e) = 1 \right\}} $ (that is, $P$ is the number of private 2-blocks).

By Lemma \ref{Lemma 3'}, if a 4-block in $B$ has a public 3-block, then it has a private 2-block. The contrapositive is of use to us: if a 4-block $m\in B$ has no private 2-block, then every 3-block of $m$ is private.

Now, we enumerate over private 2-blocks; say, $\{x,y\}$ is private to $m=\{a,b,x,y\}\in B$. Since $\{a,b,x\}$ and $\{a,b,y\}$ are the only 3-blocks of $m$ not containing the pair $\{x,y\}$, these are the only possible public 3-blocks of $m$.

In other words, a 4-block in $B$ containing a private 2-block has at most 2 public 3-blocks. We can bound the number of pairs $(A, X)$ where $A \in B$ and $X$ is a public 3-block of $A$: it is at most $2P$.

Thus,
\[E\leq \sum\limits_{\substack{l\in \binom{[n]}{3}\\ \lambda(l) \geq 2}} \lambda(l)\leq 2P\]

We need one more inequality. For each private $\{x,y\}$ of a 4-block $m$, any choice of a third point $z\in [n]\setminus m$ determines an uncovered 3-block. Thus each private 2-block is contained in exactly $n-4$ uncovered 3-blocks. Hence, each of the $P$ private 2-blocks can be used to construct $n-4$ uncovered 3-blocks. And given that an uncovered 3-block has at most 3 private 2-blocks which could be used to construct it in this way, we have that the $P$ private 2-blocks define at least $P(n-4)/3$ distinct uncovered 3-blocks. That is, $U\geq P(n-4)/3$.

Now, we combine our inequalities:

\[4\abs B = {n\choose 3} - U + E\leq {n\choose 3} - \dfrac{P(n-4)}{3} +2P = {n\choose 3} - \dfrac{P(n-10)}{3}.\]

In particular, for $n > 10$, we have $\abs B \leq {n\choose 3}/4$, with equality holding only if $P=0$.

Of course, if $\abs B = \binom{n}{3}/4$ and $P=0$, that further implies $E=U=0$. Then each 3-block is covered exactly once; we conclude in this case that $([n],B)$ is an $\mathrm{SQS}(n)$.
\end{proof}

\begin{corollary}\label{Cor 5'}
Suppose that $B\subseteq \binom{[n]}{4}$, with $n>10$, and let $p={n\choose 3}/4$. Then, $([n],B)$ is an $S(3,4,n)$ if and only if $\betti_i(S/M_B^{(2)}) = \binom{p}{i}$, for all $i=0,\dots,5$.
\end{corollary}

\begin{proof}
The forward implication follows from Proposition \ref{Prop 2'}. For the converse, suppose that $S/M_B^{(2)}$ has the Betti numbers given in the statement of this corollary. Then, we must have that $\abs B = \binom{n}{3}/4$ and $B^{(2)}$ is 5-coverfree (and hence, 4-coverfree). Then, by Theorem \ref{Th 4'}, $([n], B)$ is an $S(3,4,n)$.
\end{proof}

\section{Characterization of $S(t,k,n)$}

The question of whether Steiner systems $S(t,k,n)$ exist for large values of $t$ was posed by various mathematicians of the 1800s, and was answered in the affirmative by Peter Keevash in 2014 [Ke]. In this section we characterize such systems. We start with a result analogous to Lemma \ref{Lemma 3'}.

\begin{lemma}\label{Lemma 6'}
For $2 \leq t < k < n$, let $B\subseteq \binom{[n]}{k}$ and suppose that $B^{(t-1)}$ is $\left(\binom{k}{t-1} - t+1\right)$-coverfree. If $m,m'$ are distinct $k$-blocks of $B$ and $\abs{m\cap m'} \geq t$, then $m$ has a private $(t-1)$-block.
\end{lemma}

\begin{proof}
By means of contradiction, suppose that $m$ has no private $(t-1)$-blocks. Note that $\abs*{\binom{m\cap m'}{t-1}} \geq t$, so

\[\abs*{\binom{m}{t-1} \setminus \binom{m'}{t-1}} = \abs*{\binom{m}{t-1}} - \abs*{\binom{m'\cap m}{t-1}} \leq \binom{k}{t-1}-t. \]

If $l$ is a $(t-1)$-block of $\binom{m}{t-1}\setminus \binom{m'}{t-1}$, there exists $m_l\in B\setminus \{m\}$ such that $l\subseteq m_l$, as $m$ has no private $(t-1)$-blocks.

Hence,

\[\binom{m}{t-1}\subseteq \binom{m'}{t-1}\cup \bigcup\limits_{l\in \binom{m}{t-1}\setminus \binom{m'}{t-1}} \binom{m_l}{t-1},\]

which means that $\binom{m}{t-1}$ is covered by the union of at most $\binom{k}{t-1} - t+1$ elements in $B^{(t-1)}$, contradicting $\left(\binom{k}{t-1} - t+1\right)$-coverfreeness.
\end{proof}

\begin{theorem}\label{Th 7'}
Let $B\subseteq\binom{[n]}{k}$ and suppose that $B^{(t-1)}$ is $\left(\binom{k}{t-1} - t+1\right)$-coverfree. Then for $n>k+t\left(\binom{k}{t} -k+t-1\right)$, we have that $\abs B \leq \binom{n}{t}/\binom{k}{t}$, with equality holding only if $([n],B)$ is an $S(t,k,n)$.
\end{theorem}

\begin{proof}
For each $l\subseteq [n]$, let $\lambda(l) = \abs{\{m\in B:l\subseteq m\}}$. Now define 
\[
U = \abs*{\left\{ l\in \binom{[n]}{t}:\lambda(l) = 0 \right\}}, 
\qquad 
E = \sum\limits_{l \in \binom{[n]}{t}} \max(\lambda(l)-1,0).
\]
Finally, let $P = \abs*{\Bigl\{ e\in \binom{[n]}{t-1}:\lambda(e) = 1 \Bigr\}}$.

If $e$ is a private $(t-1)$-block and $m\in B$ is the unique $k$-block containing $e$, then the $t$-blocks in ${m\choose t}$ that contain $e$ are of the form $e\cup\{x\}$, with $x\in m\setminus e$. There are exactly $k-t+1$ of these. It follows that $m$ has ${k\choose t}-(k-t+1)$ $t$-blocks which do not contain $e$ and thus, $m$ has at most ${k\choose t}-(k-t+1)$ public $t$-blocks.

On the other hand, if $m\in B$ has no private $(t-1)$-blocks, then every $t$-block of $m$ is private by the contrapositive of Lemma \ref{Lemma 6'}.

Therefore, the number of $k$-blocks $m\in B$ having public $t$-blocks is at most the number of private $(t-1)$-blocks, that is, $P$.

We conclude that the number of pairs $(m, l)$ where $m \in B$, and $l$ is a public $t$-block of $m$ is at most $P\left({k\choose t}-(k-t+1)\right)$. That is,
\[E\leq \sum\limits_{\substack{l\in \binom{[n]}{t} \\ \lambda(l) \geq 2}} \lambda(l)\leq P\left({k\choose t}-(k-t+1)\right).\]
						
We need one more inequality. If $e$ is a private $(t-1)$-block and $m\in B$ is the unique $k$-block containing $e$, then each $t$-block of the form $e\cup\{x\}$, with $x\in [n]\setminus m$ is uncovered. There are $n-k$ of these. Counting over all private $(t-1)$-blocks $e$, we obtain $P(n-k)$ uncovered $t$-blocks. Some uncovered $t$-block may be repeated, of course, but since a $t$-block has exactly $t$ different $(t-1)$-blocks, the number of distinct uncovered $t$-blocks must be at least $P(n-k)/t$. That is, $U\geq P(n-k)/t$. 
						
Combining these bounds, we obtain
\begin{align*}
    {k\choose t} \abs B &= {n\choose t} - U + E\leq {n\choose t} - \frac{P(n-k)}{t} + P\left({k\choose t}-k+t-1)\right) \\
    &= {n\choose t} - P\left(\frac{n-k}{t} - {k\choose t} + k - t + 1\right)
\end{align*}

Hence, for $n>k+t\left({k\choose t}-k+t-1\right)$, we have $\binom{k}{t} \abs B \leq \binom{n}{t}$, which means that $\abs B \leq \binom{n}{t}/\binom{k}{t}$, with equality only if $P=0$. But if $P=0$ and $\abs B = \binom{n}{t}/\binom{k}{t}$, then in fact $E = U = 0$. Thus, each $t$-block is covered exactly once, so we conclude in this case that $([n],B)$ is an $S(t,k,n)$.
\end{proof}

\begin{corollary}\label{Cor 8'}
Suppose that $B\subseteq {{[n]}\choose k}$, with $n>k+t\left(\binom{k}{t} - k+t-1\right)$ and let $p=\binom{n}{t}/\binom{k}{t}$. Then, $([n],B)$ is an $S(t,k,n)$ if and only if $\betti_i(S/M_B^{(t-1)}) = {p\choose i}$, for all $i=0,\ldots,{k\choose {t-1}} - 1$.
\end{corollary}

\begin{proof}
The forward implication follows from Proposition \ref{Prop 2'}. For the converse, suppose that $S/M_B^{(t-1)}$ has the Betti numbers given in the statement of this corollary. Then, we must have that $\abs B = {n\choose t}/{k\choose t}$ and $B^{(t-1)}$ is $\left({k\choose {t-1}}-1\right)$-coverfree (and hence, $\left({k\choose {t-1}}-t+1\right)$-coverfree). Then, by Theorem \ref{Th 7'}, $([n],B)$ is an $S(t,k,n)$.
\end{proof}

\bigskip

\noindent \textbf{Acknowledgements}: This work was supported by NSF grant DMS-2244020. The authors are grateful to the Department of Mathematics and Statistics at California State University, Chico for providing the setting where the Research Experiences for Undergraduates (REU) program took place.

\end{document}